\documentclass[10pt]{amsart}

\usepackage{hyperref}
\hypersetup{colorlinks=true, urlcolor=blue, citecolor=blue, linkcolor=blue}

\usepackage{amsmath,amsthm, amssymb}

\usepackage{graphicx,mathrsfs,tikz,fancyhdr}

\usepackage{thmtools}
\usepackage{nameref}
\usepackage[backend=bibtex]{biblatex}

\bibliography{Masseybib.bib}

\newcommand{\coker}{{\operatorname{coker}}}
\newcommand{\U}{{\mathcal U}}
\newcommand{\0}{{\mathbf 0}}
\newcommand{\C}{{\mathbb C}}
\newcommand{\Z}{{\mathbb Z}}

\newcommand{\Q}{{\mathbb Q}}

\newcommand{\strat}{{\mathfrak S}}

\newcommand{\supp}{\operatorname{supp}}

\newcommand{\im}{\mathop{\rm im}\nolimits}

\newcommand{\Cdot}{\mathbf C^\bullet}
\newcommand{\Idot}{\mathbf I^\bullet}

\newcommand{\Pdot}{\mathbf P^\bullet}
\newcommand{\Qdot}{\mathbf Q^\bullet}

\newcommand{\Kdot}{\mathbf  K^\bullet}
\newcommand{\Mdot}{\mathbf  M^\bullet}

\newtheorem{defn0}{Definition}[section]
\newtheorem{prop0}[defn0]{Proposition}
\newtheorem{conj0}[defn0]{Conjecture}
\newtheorem{thm0}[defn0]{Theorem}
\newtheorem{lem0}[defn0]{Lemma}
\newtheorem{corollary0}[defn0]{Corollary}
\newtheorem{example0}[defn0]{Example}
\newtheorem{remark0}[defn0]{Remark}
\newtheorem{question0}[defn0]{Question}

\newenvironment{defn}{\begin{defn0}}{\end{defn0}}

\newenvironment{thm}{\begin{thm0}}{\end{thm0}}

\newenvironment{cor}{\begin{corollary0}}{\end{corollary0}}
\newenvironment{exm}{\begin{example0}\rm}{\end{example0}}
\newenvironment{rem}{\begin{remark0}\rm}{\end{remark0}}

\newcommand{\thmref}[1]{Theorem~\ref{#1}}

\newcommand{\exref}[1]{Example~\ref{#1}}

\title[Kernels, Images, and Cokernels in the Perverse Category]{A Note on Kernels, Images, and Cokernels in the Perverse Category}

\author{David B. Massey}
\address{Department of Mathematics\\
  Northeastern University\\
  Boston, Massachusetts 02115}
\email[D.~Massey]{d.massey@neu.edu}
\begin{document}

\begin{abstract}
We discuss the relationship between kernels, images and cokernels of morphisms between perverse sheaves and  induced maps on stalk cohomology.\end{abstract}

\maketitle

\thispagestyle{fancy}

\lhead{}
\chead{}
\rhead{ }

\lfoot{}
\cfoot{}
\rfoot{}

\section{Introduction}

General references for the derived category and perverse sheaves are \cite{kashsch}, \cite{dimcasheaves}, and \cite{inthom2}. As we are always considering the derived category, we follow the usual practice of omitting the ``R''s in front of right derived functors.

\smallskip

We fix a complex analytic space $X$ and a base ring, $R$, which is a commutative, regular, Noetherian ring, with finite Krull dimension (e.g., $\Z$, $\Q$, or $\C$).  We let $\operatorname{Perv}(X)$ denote the Abelian category of perverse sheaves of $R$-modules on $X$.

Suppose that $\Pdot$ and $\Qdot$ are objects in $\operatorname{Perv}(X)$, and $T$ is a morphism between them. We wish to look the kernels, images, and cokernels of $T$ in $\operatorname{Perv}(X)$ versus those of the induced maps on the stalk cohomology. The relationship between these -- or seeming lack thereof -- is well-known to experts.

\smallskip

Consider the following simple, but illustrative, example.

\smallskip

\begin{exm}\label{exm:fund} Let $X$ be the union of the coordinate axes in $\C^2$; then, the shifted constant sheaf $\mathbf R_X^\bullet[1]$ is perverse. Let $f$ denote the resolution of singularities map from two disjoint complex lines to $X$.  Then, $\mathbf I_X^\bullet:=f_*f^*\mathbf R_X^\bullet[1]$ is also perverse (it is, in fact, the intersection cohomology sheaf with constant coefficients on $X$).

The perverse sheaf $\mathbf I_X^\bullet$ is easy to describe; it is the direct sum of the extensions by zero of the shifted constant sheaves on each of the axes. In particular, the stalk cohomology at $x\in X$ of $\mathbf I_X^\bullet$ is zero outside of degree $-1$ and, in degree $-1$, is $R$ at $x\neq\0$ and is $R\oplus R$ at $x=\0$.

There is a natural morphism $T: \mathbf R_X^\bullet[1]\rightarrow f_*f^*\mathbf R_X^\bullet[1]=\mathbf I_X^\bullet$ which is easy to describe on the level of stalks. Note that, in what follows, the superscript $-1$ is for the degree, {\bf not} for the inverse map. For each $x\in X$, the map 
$$T_x^{-1}: H^{-1}(\mathbf R_X^\bullet[1])_x\rightarrow H^{-1}(\mathbf I_X^\bullet)_x$$
is the diagonal map, i.e., the identity for $x\neq\0$ and the diagonal map $R\rightarrow R\oplus R$ when $x=\0$.

Thus, in all degrees, for all $x\in X$, $\ker T^i_x=0$. However, we claim that $\ker T\neq 0$ in $\operatorname{Perv}(X)$.

\medskip

How do you see this?

\medskip

Consider the mapping cone $\Mdot$ of $T: \mathbf R_X^\bullet[1]\rightarrow \mathbf I_X^\bullet$ in the derived category, so that we have a distinguished triangle
$$
\mathbf R_X^\bullet[1]\xrightarrow{T} \mathbf I_X^\bullet\rightarrow \Mdot\xrightarrow{[1]} \mathbf R_X^\bullet[1].
$$

\smallskip

\noindent The complex $\Mdot$ is easy to describe; it is supported only at the origin, and has non-zero stalk cohomology only in degree $-1$, where it is the cokernel of the diagonal map $R\rightarrow R\oplus R$, i.e., is isomorphic to $R$.

The complex $\Mdot$ is not perverse; at an isolated point in the support of a perverse sheaf, the stalk cohomology can be non-zero only in degree $0$. However, we may ``turn the triangle'' to obtain the distinguished triangle
$$
\Mdot[-1]\xrightarrow{S}\mathbf R_X^\bullet[1]\xrightarrow{T} \mathbf I_X^\bullet\xrightarrow{[1]}\Mdot[-1],
$$
where $\Mdot[-1]$ {\bf is} a perverse sheaf. Now distinguished triangles of perverse sheaves are the short exact sequences in  $\operatorname{Perv}(X)$. Therefore, in $\operatorname{Perv}(X)$, we have a short exact sequence
$$
0\rightarrow\Mdot[-1]\xrightarrow{S}\mathbf R_X^\bullet[1]\xrightarrow{T} \mathbf I_X^\bullet\rightarrow 0,
$$
and so $\ker T\cong \Mdot[-1]\neq 0$ even though the kernels of all of the induced maps on stalk cohomology are zero.

Furthermore, $T$ is a surjection, even though the induced maps on the stalk cohomology at $\0$ are not surjections in all degrees, and the map $S$ is an injection, even though $S$ induces the zero map on the stalk cohomology in all degrees.

\end{exm}

\bigskip

In this note, we wish to clarify what is going on here by looking at induced maps on stalk cohomology after applying the shifted vanishing cycle functor. We also show that, if the base ring is a field and $T$ is an endomorphism on a perverse sheaf, then there is a nice relationship between the kernels, images, and cokernels of the induced maps on stalk cohomology and the perverse kernel, image, and cokernel of $T$.

\bigskip

\section{Enter the vanishing cycles}

We want to analyze kernels, images, and cokernels in $\operatorname{Perv}(X)$ by looking at stalks  and homomorphisms of modules. As we shall see, we can do this if we first take vanishing cycles supported at isolated points.

\smallskip

\begin{exm}Let us look again at the map from \exref{exm:fund}, where we had $\mathbf R_X^\bullet[1]\xrightarrow{\ T\ } \mathbf I_X^\bullet$, where $X$ was the union of the $x$- and $y$-axes in $\C^2$. We wish to see that taking vanishing cycles can explain why $\ker T$ must not be zero.

Let $L$ be the restriction to $X$ of the linear function $\hat L(x,y)=x+y$. Then the shifted vanishing cycles $\phi_L[-1]$ along $L$ is an exact functor from $\operatorname{Perv}(X)$ to $\operatorname{Perv}(V(L))=\operatorname{Perv}(\{\0\})$, and a perverse sheaf on an isolated point consists, up to isomorphism, of a finitely-generated $R$-module in degree zero -- its stalk cohomology at the point --and zeroes in all other degrees.

Let $\Kdot:=\ker T$ and $\Cdot:=\coker\, T$, so that we have an exact sequence in  $\operatorname{Perv}(X)$:
$$
0\rightarrow \Kdot\rightarrow \mathbf R_X^\bullet[1]\xrightarrow{\ T\ } \mathbf I_X^\bullet\rightarrow\Cdot\rightarrow 0.
$$

Let $T_L:=\phi_L[-1](T)$ and let $T^0_{L,\0}$ denote the homomorphism induced by $T_L$ on the stalk cohomology at $\0$ in degree $0$.

As $\phi_L[-1]$ is an exact functor, we have an exact sequence of perverse sheaves
$$
0\rightarrow \phi_L[-1]\Kdot\rightarrow \phi_L[-1]\mathbf R_X^\bullet[1]\xrightarrow{T_L} \phi_L[-1]\mathbf I_X^\bullet\rightarrow\phi_L[-1]\Cdot\rightarrow 0,
$$
which are all supported at just the origin; so this exact sequence corresponds to the sequence of $R$-modules:
$$0\rightarrow H^0\left(\phi_L[-1]\Kdot\right)_\0\rightarrow H^0\left(\phi_L[-1]\mathbf R_X^\bullet[1]\right)_\0\xrightarrow{T^0_{L,\0}}H^0\left(\phi_L[-1]\mathbf I_X^\bullet\right)_\0\rightarrow H^0\left(\phi_L[-1]\Cdot\right)_\0\rightarrow 0.$$

\smallskip

Now, $H^0\left(\phi_L[-1]\mathbf R_X^\bullet[1]\right)_\0\cong \Z$, while $H^0\left(\phi_L[-1]\mathbf I_X^\bullet\right)_\0=0$. Consequently, $H^0\left(\phi_L[-1]\Kdot\right)_\0\cong \Z$, and so $\Kdot$ cannot be the zero complex.
\end{exm}

\medskip

We return to the case of a general analytic space $X$ and a morphism $\Pdot\xrightarrow{T}\Qdot$. We will use the vanishing cycles along various functions to analyze the situation. First we need two definitions.

\begin{defn} We define an ordering on (isomorphism classes of) Noetherian $R$-modules by: $M\leq N$ if and only if there exists an $R$-module $P$ such that $M\oplus P\cong N$.
\end{defn}

Note that reflexivity and transitivity are immediate. Anti-symmetry reduces to proving the weak cancellation property that $M\oplus P\cong M$ implies that $P=0$, provided that $M$ is Noetherian; this is an easy exercise.

\medskip

\begin{defn} Let $x\in X$, let $\U$ be an open neighborhood of $x$ in $X$, and let $f:(\U, x)\rightarrow (\C,0)$ such that $\dim_x\supp\phi_f[-1]\big(\Pdot_{|_\U}\big)\leq 0$ and $\dim_x\supp\phi_f[-1]\big(\Qdot_{|_\U}\big)\leq 0$.

Then we will say that $f$ is $(\Pdot, \Qdot)$-{\bf isolating} at $x$, and we let $T_{f, x}^0$ denote the map induced by $\phi_f[-1]\big(T_{|_\U}\big)$ from $H^0\big(\phi_f[-1](\Pdot_{|_\U})\big)_x$ to $H^0\big(\phi_f[-1](\Qdot_{|_\U})\big)_x$. As this depends only on the germ of $f$ at $x$, we will suppress the explicit restrictions to $\U$ below.

\end{defn}

The cases above where $\dim_x<0$ are meant to allow for the possibility that the supports are ``empty at $x$'', i.e., the cases where $x$ is not in the supports. Note that, for all $x\in X$, there exists an $f$ which is $(\Pdot, \Qdot)$-isolating at $x$ since we may select a common Whitney stratification $\strat$ of $X$ with respect to which both $\Pdot$ and $\Qdot$ are constructible, and then, if $S$ is the stratum containing $x$, take $f$ to have a nondegenerate critical point at $x$ with respect to $\strat$ in the sense of \cite{stratmorse}, 6.A.2.

\bigskip

Now we can state:

\begin{thm}\label{thm:vanthm}
Given a morphism $T:\Pdot\rightarrow\Qdot$, we have:

\begin{enumerate}
\item
$$
\supp(\ker T)=\overline{\{x\in X |  \textnormal{ there exists a $(\Pdot, \Qdot)$-isolating } f \textnormal{ at } x \textnormal{ such that } \ker T_{f, x}^0\neq 0\}}.
$$
\smallskip
\item Suppose that $Y$ is an irreducible component of $\supp(\ker T)$, and let $d:=\dim Y$. Then, for a generic point $x\in Y$, $H^*(\ker T)_x$ is zero, except in degree $-d$, where we have
$$
H^{-d}(\ker T)_x\cong M_{\operatorname{min},x},
$$
where $M_{\operatorname{min},x}$ is the minimum non-zero module which occurs as $\ker\big(T_{f, x}^0\big)$ for some $(\Pdot, \Qdot)$-isolating $f$ at $x$.
\end{enumerate}

\medskip

Furthermore, the above statements remain true if each instance of $\ker$ is replaced by $\coker$, or if each instance of $\ker$ is replaced by $\im$.
\end{thm}
\begin{proof} The proofs for kernels, images, and cokernels are very similar;  we prove the kernel and image statements, and leave the cokernel statements as exercises.

\smallskip

We define 
$$\operatorname{stalkker}_\phi:= \overline{\{x\in X |  \textnormal{ there exists a $(\Pdot, \Qdot)$-isolating } f \textnormal{ at } x \textnormal{ such that } \ker T_{f, x}^0\neq 0\}},$$  
and
$$\operatorname{stalkim}_\phi:= \overline{\{x\in X |  \textnormal{ there exists a $(\Pdot, \Qdot)$-isolating } f \textnormal{ at } x \textnormal{ such that } \im T_{f, x}^0\neq 0\}}.$$

\smallskip

Let $\Kdot:=\ker T$, $\Idot:=\im T$, and $\Cdot:=\coker\, T$, so that we have two short exact sequences of perverse sheaves
\begin{equation}\tag{$\dagger$}
0\rightarrow\Kdot\rightarrow \Pdot\xrightarrow{\alpha}\Idot\rightarrow 0\ \ \textnormal{ and }\ \ 0\rightarrow\Idot\xrightarrow{\beta} \Qdot\rightarrow\Cdot\rightarrow 0,
\end{equation}
where $T=\beta\circ\alpha$; of course, we also have the combined exact sequence
\begin{equation}\tag{$\ddagger$}
0\rightarrow\Kdot\rightarrow \Pdot\xrightarrow{T} \Qdot\rightarrow\Cdot\rightarrow 0.
\end{equation}

\bigskip

\noindent{\bf kernel statements}:

\smallskip

We shall first show that $\operatorname{stalkker}_\phi\subseteq \supp\Kdot$ by showing that $X-\supp(\Kdot)=X-\operatorname{stalkker}_\phi$.

\smallskip

Suppose that $p\in X-\supp\Kdot$, i.e., suppose that $p\in X$ possesses an open neighborhood $\U$ such that, for all $x\in\U$, $H^*(\Kdot)_x=0$, i.e., $\Kdot=0$ on $\U$. Restricting to $U$, we are reduced to the case where ($\ddagger$) becomes the short exact sequence
$$
0\rightarrow\Pdot\xrightarrow{T} \Qdot\rightarrow\Cdot\rightarrow 0.
$$
Suppose that $f$ is $(\Pdot, \Qdot)$-isolating at a point $x\in \U$. Then, as $\phi_f[-1]$ is an exact functor, we have a short exact sequence
$$
0\rightarrow H^0(\phi_f[-1]\Pdot)_x\xrightarrow{T^0_{f,x}} H^0(\phi_f[-1]\Qdot)_x\rightarrow H^0(\phi_f[-1]\Cdot)_x\rightarrow 0.
$$
Thus, $\ker\big(T_{f, x}^0\big)=0$; this shows that $p\in X-\operatorname{stalkker}_\phi$. Hence, we have proved that $$X-\supp(\Kdot)\subseteq  X-\operatorname{stalkker}_\phi,$$
i.e., that $\operatorname{stalkker}_\phi\subseteq \supp(\Kdot)$.

\medskip

We must show the reverse containment; in fact, we shall prove Item 2 of the theorem at the same time.

\medskip

Suppose that $p\in \supp(\Kdot)$. Let $\strat$ be a Whitney stratification, with connected strata, of $X$, with respect to which $\Pdot$, $\Qdot$, $\Kdot$, and $\Cdot$ are all constructible. Then $p$ must be in the closure of a maximal stratum $S$ (ordered in the standard way by inclusion in the closure) of $\supp(\Kdot)$. 

Let $x\in S$ and let $d:=\dim S$. Let $f$ be a function from an open neighborhood of $x$ in $X$ to $\C$ with a complex nondegenerate critical point at $x$ with respect to $\strat$ (in the sense of \cite{stratmorse}, 6.A.2). Then $f$ is $(\Pdot, \Qdot)$-isolating at $x$, and $H^i(\phi_f[-1]\Kdot)_x$ is zero if $i\neq 0$, while $H^0(\phi_f[-1]\Kdot)_x\cong H^{-d}(\Kdot)_x\neq 0$; note that $H^i(\phi_f[-1]\Kdot)_x$ is a minimum among such non-zero modules as $f$ has a nondegenerate critical point at $x$.

Now, we have the exact sequence
$$0\rightarrow H^0\big(\phi_f[-1]\Kdot\big)_x\rightarrow H^0\big(\phi_f[-1]\Pdot\big)_x\xrightarrow{T^0_{f, x}} H^0\big(\phi_f[-1]\Qdot\big)_x\rightarrow H^0\big(\phi_f[-1]\Cdot\big)_x\rightarrow 0.
$$
Therefore, $\ker T^0_{f,x}\cong H^0\big(\phi_f[-1]\Kdot\big)_x\neq 0$, and we have shown that $p\in \operatorname{stalkker}_\phi$, i.e., that $\supp(\Kdot)\subseteq\operatorname{stalkker}_\phi$.

\bigskip

\noindent{\bf image statements}:

\smallskip

We shall first show that $\operatorname{stalkim}_\phi\subseteq \supp\Idot$ by showing that $X-\supp(\Idot)=X-\operatorname{stalkim}_\phi$.

\smallskip

Suppose that $p\in X-\supp\Idot$, i.e., suppose that $p\in X$ possesses an open neighborhood $\U$ such that, for all $x\in\U$, $H^*(\Idot)_x=0$, i.e., $\Idot=0$ on $\U$. Restricting to $U$, we are reduced to the case where ($\ddagger$) becomes the exact sequence
$$
0\rightarrow\Kdot\rightarrow \Pdot\xrightarrow{0} \Qdot\rightarrow\Cdot\rightarrow 0.
$$
Suppose that $f$ is $(\Pdot, \Qdot)$-isolating at a point $x\in \U$. Then, it follows immediately that $T^0_{f,x}=0$, i.e., $\im\big(T_{f, x}^0\big)=0$. This shows that $p\in X-\operatorname{stalkim}_\phi$. Hence, we have proved that $$X-\supp(\Idot)\subseteq  X-\operatorname{stalkim}_\phi,$$
i.e., that $\operatorname{stalkim}_\phi\subseteq \supp(\Idot)$.

\medskip

We must show the reverse containment; in fact, we shall again prove Item 2 of the theorem at the same time.

\medskip

Suppose that $p\in \supp(\Idot)$. Let $\strat$ be a Whitney stratification, with connected strata, of $X$, with respect to which $\Pdot$, $\Qdot$, $\Kdot$, $\Idot$, and $\Cdot$ are all constructible. Then $p$ must be in the closure of a maximal stratum $S$ of $\supp(\Idot)$. 

Let $x\in S$ and let $d:=\dim S$. Let $f$ be a function from an open neighborhood of $x$ in $X$ to $\C$ with a complex nondegenerate critical point at $x$ with respect to $\strat$ (in the sense of \cite{stratmorse}, 6.A.2). Then $f$ is $(\Pdot, \Qdot)$-isolating at $x$, and $H^i(\phi_f[-1]\Idot)_x$ is zero if $i\neq 0$, while 
$$H^0(\phi_f[-1]\Idot)_x\cong H^{-d}(\Idot)_x\neq 0;$$
 note, as before, that $H^i(\phi_f[-1]\Idot)_x$ is a minimum among such non-zero modules as $f$ has a nondegenerate critical point at $x$.

Now, letting $\alpha^0_{f,x}$ and $\beta^0_{f,x}$ denote the maps induced on the stalk cohomology in degree $0$ at $x$ by the maps $\phi_f[-1](\alpha)$ and $\phi_f[-1](\beta)$, respectively, we have the short exact sequences
$$0\rightarrow H^0\big(\phi_f[-1]\Kdot\big)_x\rightarrow H^0\big(\phi_f[-1]\Pdot\big)_x\xrightarrow{\alpha^0_{f,x}}H^0\big(\phi_f[-1]\Idot\big)_x\rightarrow 0$$
and 
$$0\rightarrow H^0\big(\phi_f[-1]\Idot\big)_x\xrightarrow{\beta^0_{f,x}} H^0\big(\phi_f[-1]\Qdot\big)_x\rightarrow H^0\big(\phi_f[-1]\Cdot\big)_x\rightarrow 0,$$
where $T^0_{f,x}=\beta^0_{f,x}\circ\alpha^0_{f,x}$.

As $\alpha^0_{f,x}$ is a surjection and $\beta^0_{f,x}$ is an injection, we see that 
$$
\im T^0_{f,x}=\im (\beta^0_{f,x}\circ\alpha^0_{f,x})\cong \im (\beta^0_{f,x})\cong H^0\big(\phi_f[-1]\Idot\big)_x\cong H^{-d}(\Idot)_x\neq 0.
$$

Therefore, we have shown that $p\in \operatorname{stalkim}_\phi$, i.e., that $\supp(\Idot)\subseteq\operatorname{stalkim}_\phi$.
\end{proof}

\bigskip

\begin{cor} The morphism $T:\Pdot\rightarrow\Qdot$ is an injection (resp., zero morphism, surjection) if and only if, for all $x\in X$, for all $(\Pdot, \Qdot)$-isolating $f$ at $x$, $T^0_{f,x}$ is an injection (resp., zero morphism, surjection).
\end{cor}
\begin{proof} This is immediate from Item 1 of the theorem in the kernel, image, and cokernel cases.
\end{proof}

\bigskip

\section{The Special Case of an Endomorphism over a Field}

It may seem strange, but -- in the case of an endomorphism where the base ring is a field -- we do not need to apply vanishing cycles in order to obtain a result along the lines of \thmref{thm:vanthm}, though we must drop the conclusion about images.

\smallskip

We will use properties of characteristic cycles (see \cite{kashsch}, \cite{dimcasheaves}, \cite{schurbook}):

\begin{itemize}

\item Characteristic cycles $\operatorname{CC}(\Pdot)$ of complexes of sheaves are additive over distinguished triangles; in particular, characteristic cycles are additive over exact sequences in $\operatorname{Perv}(X)$. 

\smallskip

\item  For perverse sheaves $\Pdot$ with a field for a base ring, the subset $X$ which underlies $\operatorname{CC}(\Pdot)$ equals $\supp\Pdot$; in particular,  in $\operatorname{CC}(\Pdot)=0$, then $\Pdot=0$.

\end{itemize}

\medskip

\begin{thm}\label{thm:endothm} If the base ring is a field, and we have an endomorphism $T:\Pdot\rightarrow\Pdot$, then:
\begin{enumerate}
\item
$$
\supp(\ker T)=\overline{\{x\in X |   \ker T_x^*\neq 0\}}.
$$
\smallskip
\item Suppose that $Y$ is an irreducible component of $\supp(\ker T)$, and let $d:=\dim Y$. Then, for a generic point $x\in Y$, $H^*(\ker  T)_x$ is zero, except in degree $-d$, where we have
$$
H^{-d}(\ker T)_x\cong \ker \big(T^{-d}_x\big).
$$

\end{enumerate}

\medskip

Furthermore, the above statements remain true if each instance of $\ker$ is replaced by $\coker$.
\end{thm}
 \begin{proof} We shall prove the statements about kernels; the cokernel proof is completely analogous. Let $\Kdot:=\ker T$ and $\Cdot:=\coker\, T$, and $\Idot=\im T$.
 
 \smallskip
 
 Suppose that $p\in X-\supp(\Kdot)$. Then there exists an open neighborhood $\U$ of $p$ such that the restriction of $\Kdot$ to $\U$ is zero. Thus, restricting to $\U$, we have a short exact sequence in $\operatorname{Perv}(X)$:
 $$
 0\rightarrow \Pdot\xrightarrow{T}\Pdot\rightarrow\Cdot\rightarrow 0.
 $$
 But now the properties of characteristic cycles -- the additivity and support properties -- imply instantly that $\Cdot=0$. Therefore, $T$ is an isomorphism, i.e., induces isomorphisms on the stalk cohomology at every point. And so, for all $x\in \U$, for all degrees $i$, $\ker T_x^i=0$. Consequently, $p\in X-\overline{\{x\in X |   \ker T_x^*\neq 0\}}$. Thus, we have shown that $\overline{\{x\in X |   \ker T_x^*\neq 0\}}\subseteq \supp(\Kdot)$.
 
 \medskip
 
We must show the reverse containment; in fact, we shall again prove Item 2 of the theorem at the same time.

\medskip

We have the canonical exact sequence
$$
0\rightarrow\Kdot\rightarrow\Pdot\xrightarrow{T}\Pdot\rightarrow\Cdot\rightarrow 0.
$$
Once again, using properties of characteristic cycles, we see that $\operatorname{CC}(\Kdot)=\operatorname{CC}(\Cdot)$, and so $\supp\Kdot=\supp\Cdot$.

Suppose that $p\in \supp(\Kdot)$. Let $\strat$ be a Whitney stratification, with connected strata, of $X$, with respect to which $\Pdot$, $\Kdot$, and $\Cdot$ are all constructible. Then $p$ must be in the closure of a maximal stratum $S$ of  $\supp(\Kdot)=\supp(\Cdot)$.

Let $x\in S$ and let $d:=\dim S$. Then, $H^*(\Kdot)_x$ and $H^*(\Cdot)_x$ are both non-zero precisely in degree $-d$.

From the short exact sequence
$$
0\rightarrow\Idot\xrightarrow{\beta}\Pdot\rightarrow\Cdot\rightarrow 0,
$$
and using that $H^{-d-2}(\Cdot)_x=H^{-d-1}(\Cdot)_x=0$, we conclude that:

\begin{itemize}
\item $H^{-d-1}(\Idot)_x\cong H^{-d-1}(\Pdot)_x$, and 
\smallskip
\item $\beta^{-d}_x$ from $H^{-d}(\Idot)_x$ to $H^{-d}(\Pdot)_x$ is an injection.
\end{itemize}

From the short exact sequence
$$
0\rightarrow\Kdot\rightarrow\Pdot\xrightarrow{\alpha}\Idot\rightarrow 0,
$$
and using that $H^{-d-1}(\Kdot)_x=H^{-d+1}(\Kdot)_x=0$, we conclude that we have an exact sequence
$$
0\rightarrow H^{-d-1}(\Pdot)_x\rightarrow H^{-d-1}(\Idot)_x\rightarrow  H^{-d}(\Kdot)_x\rightarrow H^{-d}(\Pdot)_x\rightarrow H^{-d}(\Idot)_x\rightarrow 0.
$$
However, as we saw above, $H^{-d-1}(\Idot)_x$ and $H^{-d-1}(\Pdot)_x$ are isomorphic finite-dimensional vector spaces and, consequently, the injection on the left of the above exact sequence is an isomorphism. 

Therefore, we have a short exact sequence
$$
0\rightarrow  H^{-d}(\Kdot)_x\rightarrow H^{-d}(\Pdot)_x\xrightarrow{\alpha^{-d}_x} H^{-d}(\Idot)_x\rightarrow 0.
$$
As $\beta^{-d}_x$ is an injection, we conclude that
$$
\ker T^{-d}_x=\ker (\beta^{-d}_x\circ\alpha^{-d}_x)=\ker \alpha^{-d}_x\cong H^{-d}(\Kdot)_x\neq 0.
$$
This concludes the proof.
\end{proof}

\medskip

\begin{rem} Our primary interest in \thmref{thm:endothm} centers around eigenvalues of the monodromy for the nearby and vanishing cycles.

\smallskip

Suppose that our base ring is $\C$. Let $\Qdot$ by a perverse sheaf on a complex analytic space $Y$ and suppose that we have a complex analytic $f:Y\rightarrow\C$. Then, we have the perverse sheaves $\psi_f[-1]\Qdot$ and $\phi_f[-1]\Qdot$ on $X:=V(f)$, together with their respective monodromy automorphisms, $T_f$ and $\widetilde T_f$.

For each $\lambda\in\C$, there is the question: what does it mean for $\lambda$ to be an eigenvalue of $T_f$ (resp., $\widetilde T_f$)? Does it mean that $\ker(\lambda\cdot\operatorname{id}-T_f)\neq 0$ (resp., $\ker(\lambda\cdot\operatorname{id}-\widetilde T_f)\neq 0$) or does it mean that there is a point $x\in X$ and a degree $i$ such that, in the stalk cohomology, we have $\ker(\lambda\cdot\operatorname{id}-T^i_{f,x})\neq 0$ (resp., $\ker(\lambda\cdot\operatorname{id}-\widetilde T^i_{f,x})\neq 0$)?

\medskip

Item 1 of \thmref{thm:endothm} tells us that these conditions are equivalent.
\end{rem}

\medskip

\begin{exm}
The statement about images that one might expect to find in \thmref{thm:endothm} is simply false. 

Let $R$ be a field and consider the injection $\Mdot[-1]\xrightarrow{S}\mathbf R_X^\bullet[1]$ from \exref{exm:fund}. Define an endomorphism 
$$
\Mdot[-1]\oplus\mathbf R_X^\bullet[1]\xrightarrow{T}\Mdot[-1]\oplus\mathbf R_X^\bullet[1]
$$
by $T(a,b)=(0, S(a))$. Then the image of $T$ is isomorphic to $\Mdot[-1]$, while all of the induced maps on stalk cohomology are zero.

\smallskip

It is interesting to note that $\ker T\not\cong \coker\, T$, for $\ker T\cong \mathbf R_X^\bullet[1]$ and $\coker\, T\cong \Mdot[-1]\oplus \mathbf I_X^\bullet$, where $\mathbf I_X^\bullet$ is as in \exref{exm:fund}. And yet, 
$$\operatorname{CC}\big(\mathbf R_X^\bullet[1]\big)=\big[T^*_\0\C^2\big]+\big[T^*_{V(y)}\C^2\big]+\big[T^*_{V(x)}\C^2\big] = \operatorname{CC}\big(\Mdot[-1]\oplus \mathbf I_X^\bullet\big)
$$
(or, depending on one's shifting convention on the characteristic cycle, the characteristic cycle may be the negation of what we give).
\end{exm}

\bigskip

\printbibliography
%\bibliographystyle{plain}
%\bibliographystyle{amsalpha}
%\bibliography{Masseybib}
%\printindex
\end{document}